\documentclass{article}

\usepackage{subfigure}
\usepackage{PRIMEarxiv}
\usepackage{amsmath,amssymb,amsthm}
\usepackage{mathrsfs}
\usepackage[utf8]{inputenc} 
\usepackage[T1]{fontenc}    
\usepackage{hyperref}       
\usepackage{url}            
\usepackage{booktabs}       
\usepackage{amsfonts}       
\usepackage{nicefrac}       
\usepackage{microtype}      
\usepackage{lipsum}
\usepackage{graphicx}
\usepackage[ruled,vlined]{algorithm2e}
\usepackage[utf8]{inputenc}
\usepackage{authblk}
\usepackage{indentfirst}
\graphicspath{{media/}}     

\newcommand{\argmin}{\mathop{\arg\min}}
\newcommand{\inner}[1]{\left\langle#1\right\rangle}
\newcommand{\norm}[1]{\left\|#1\right\|}
\newcommand{\X}{\overset{\cdot}{X}}
\newcommand{\XX}{\overset{\cdot\cdot}{X}}
\newcommand{\E}{\mathcal{E}}
\newcommand{\EE}{\overset{\cdot}{\mathcal{E}}}
\newcommand{\Z}{\overset{\cdot}{Z}}
\newcommand{\Fone}[1][L]{\mathscr{F}^1_{#1}}

\newtheorem{theorem}{Theorem}
\newtheorem{lemma}{Lemma}
\newtheorem{corollary}{Corollary}

\makeatletter
\renewcommand*{\@fnsymbol}[1]{\ensuremath{\ifcase#1\or *\or \dagger\or \ddagger\or
		\mathsection\or \mathparagraph\or \|\or **\or \dagger\dagger
		\or \ddagger\ddagger \else\@ctrerr\fi}}
\makeatother

\title{Analysis Accelerated Mirror Descent via High-resolution ODEs\thanks{This work was partially supported by grant 12288201 from NSF of China}
}
\author[1]{Ya-xiang Yuan}
\author[ 1,2]{Yi Zhang\thanks{Corresponding author: zhangyi2020@lsec.cc.ac.cn}  }
\affil[1]{Institute of Computational Mathematics and Scientific/Engineering Computing, Academy of Mathematics and Systems
	Science, Chinese Academy of Sciences, Beijing 100190, China}
\affil[2]{University of Chinese Academy of Sciences, Beijing 100049, China}

\begin{document}
	\maketitle

	\begin{abstract}
		Mirror descent plays a crucial role in constrained optimization and acceleration schemes, along
		with its corresponding low-resolution ordinary differential equations (ODEs) framework have been
		proposed. However, the low-resolution ODEs are unable to distinguish between Polyak’s heavy-ball
		method and Nesterov’s accelerated gradient method. This problem also arises with accelerated mirror
		descent. To address this issue, we derive high-resolution ODEs for accelerated mirror descent and
		propose a general Lyapunov function framework to analyze its convergence rate in both continuous
		and discrete time. Furthermore, we demonstrate that accelerated mirror descent can minimize the
		squared gradient norm at an inverse cubic rate.
	\end{abstract}

	\keywords{Ordinary differential equations \and Lyapunov functions \and Gradient minimization}

	\section{Introduction}
	During the era of big data, machine learning has emerged as a crucial method for discovering information in high-dimensional datasets. Reinforcement learning, as a key branch of machine learning, frequently involves solving optimization problems within specific convex sets. However, due to the curse of dimensionality, researchers have focused on simple first-order methods in both theoretical and practical contexts.
	
	Throughout this paper, we will be considering constrained optimization problems,\begin{equation}
		\min_{x\in C} f(x),
	\end{equation}
	where $ f $ is a smooth convex function and $ C $ is a closed convex set. For unconstrained optimization problem, perhaps gradient descent is the simplest first-order method. Mirror descent proposed in \cite{Nemirovsky84} provides an efficient generalization of gradient descent to non-Euclidean geometries in constrained optimization, and has variety applications in policy gradient\cite{sutton2018} , as well as online optimization \cite{bubeck2011, cesa2006}. Fixing a step size $ s $, mirror descent is given by the recursive rule\begin{equation}
		x_{k+1}=\argmin_{x}\left\{s\inner{\nabla f(x_k),x-x_k}+\frac{D_\varphi(x,x_k)}{\sigma}\right\},
	\end{equation}
	where $ D_\varphi(\cdot,\cdot) $ represents Bergman divergence respect to a given $ \sigma- $strongly convex function $ \varphi $ on $ C $. 
	
	However, gradient descent methods always converge slowly in practice. About 40 years ago, Nesterov proposed the following accelerated gradient scheme\cite{nesterov83, nesterov18}: strating with $ x_0 $ and $ y_0=x_0 $, the iteration is \begin{equation}
		\begin{split}
			y_{k+1}&=x_k-\frac{1}{L}\nabla f(x_k),\\
			x_{k+1}&=y_{k+1}+\frac{k}{k+3}(y_{k+1}-y_k).
		\end{split}
		\label{eq:nesterovgradient}
	\end{equation}
	Where the parameter $ L $ is the Lipschitz constant of $ \nabla f $. This scheme achieves the convergence rate\[
	f(x_k)-f(x^*)\leqslant O\left(\frac{\norm{x_0-x^*}^2}{k^2}\right),
	\]
	compared to the convergence rate of gradient descent\[
	f(x_k)-f(x^*)\leqslant O\left(\frac{\norm{x_0-x^*}^2}{k}\right).
	\]
	Above $ x^* $ is the global minimizer of $ f $. Due to its fascinating theoretical results, accelerated methods has numerous extensions in convex optimization problem for unconstrained optimization\cite{nesterov12}, composite optimization\cite{beck2009, nesterov13} and stochastic optimization\cite{ghadimi2016}. In \cite{nesterov05}, Nesterov's proposed the following accelerated mirror descent\begin{equation}
		\begin{split}
			y_{k+1}&=\argmin_y\left\{\inner{\nabla f(x_k),y-x_k}+\frac{L}{2}\norm{x-x_k}^2\middle|y\in C\right\},\\
			z_{k+1}&=\argmin_z\left\{\frac{L}{\sigma}\varphi(z)+\sum_{i=0}^{k}\frac{i+1}{2}[f(x_i)+\inner{\nabla f(x_i),z-x_i}]\middle|z\in C\right\},\\
			x_{k+1}&=\frac{2}{k+3}z_{k+1}+\frac{k+1}{k+3}y_{k+1}.
		\end{split}
		\label{eq:acceleratedmirror}
	\end{equation}
	The initial feasible point $ x_0 $ is the $ \varphi $-center of $ C $, i. e. the global minimizer of $ \varphi $ over $ C $. Similarly, accelerated mirror descent admits $ O(1/k^2) $ convergence rate for $ f(y_k)-f(x^*). $
	
	There is a long history understanding the accelerated phenomenon of Nesterov's accelerated gradient method. The original idea is from estimate-sequence technique\cite{nesterov83}, which has been applied to \cite{nesterov05, nesterov12, nesterov13, tseng08, beck2009}. In recent year, Su et al.\cite{su14} consider the ODE\begin{equation}
		\XX+\frac{3}{t}\X+\nabla f(X)=0,
		\label{eq:lowode}
	\end{equation}
	for $ t>0 $ with initial $ x(0)=x_0, \X(0)=0 $, which is the continuous limit of \eqref{eq:nesterovgradient}, and can prove the convergence rate of \eqref{eq:nesterovgradient} by making use of Lyapunov function framework. Originally, Lyapunov function is used to analysis the stability of dynamic systems, see \cite{lyapunov92}, nowadays it has been generalized to analyze other accelerated first-order method, such as FISTA\cite{attouch14,attouch18,attouch16}, mirror descent\cite{krichene2015}, monotone operator\cite{ryu22}. Wibisono et al.\cite{wibisono16} shows that the trajectories of \eqref{eq:lowode} minimizes the Bergman Lagrangian and the continuous-limit of accelerated higher-order gradient method  correspond to traveling the same
	curve in spacetime at different speeds. In \cite{shi21}, Shi et al. proposed the high-resolution ODEs by introducing a gradient correction term or hessian-driven damping into \eqref{eq:lowode} and showed that the high-resolution ODEs can differentiate between Nesterov's accelerated gradient method and heavy ball method. Also they used high-resolution ODEs framework to prove that\[
	\min_{0\leqslant i\leqslant k}\norm{\nabla f(x_i)}^2\leqslant O\left(\frac{\norm{x_0-x^*}^2}{k^3}\right).
	\]
	Nowadays, high-resolution ODEs techniques have been utilized to analyze other optimization problems, which can be found in \cite{chavdarova21,lin2021}.
	\subsection{Overview of contributions}
	\begin{itemize}
		\item In section \ref{sec:2}, we establish the continuous limits of mirror descent and analyze its convergence in both continuous and discrete time settings using the Lyapunov function framework. These results provide valuable insights into extending the Lyapunov function framework to accelerated schemes.
		\item Inspired by \cite{shi21}, we derive the high-resolution ODEs for accelerated mirror descent \eqref{eq:acceleratedmirror} in section \ref{sec:3}, given by\begin{equation}
			\begin{split}
				\Z&=-\frac{t}{2}\nabla f(X),\\
				\nabla\varphi^*(Z)&=\frac{t}{2}\X+\frac{t}{2}\sqrt{s}\nabla f(X)+X.
			\end{split}
			\label{eq:highmirrorode}
		\end{equation}
		with initial $ X(0)=x_0, Z(0)\in\partial\varphi(x_0), \X(0)=0 $. Next, we exploit the Lyapunov function to prove that the solution trajectories of \eqref{eq:highmirrorode} are such that $ f(X(t))-f(x^*)\leqslant O(\frac{1}{t^2}) $. In discrete-time cases, we discuss the convergence rate of \eqref{eq:acceleratedmirror} in both unconstrained cases and constrained cases via Lyapunov function analogous to the continuous-time cases, and shows that both of them exist $ O(1/k^2) $ convergence rate for function value respect to different feasible point sequence. The cubic convergence rate of square of gradient norm in constrained cases also can be proved, and square of gradient norm like term in constrained cases, i. e. $ \norm{y_{k+1}-x_k}^2 $, converges to 0 with cubic rate. Furthermore, we construct a generalized Lyapunov function to demonstrate both the convergence rate of function values and gradient norms of higher-order mirror descent.

	\end{itemize}
	\subsection{Notations}
	We almost follow the notations of \cite{nesterov05}. Let $ E $ denotes a finite-dimensional real vector spaces, $ E^* $ be the dual space of $ E $. The value of function $ s\in E^*, x\in E $ is denoted as $ \inner{s,x}. $ The norm of $ E $ is denoted as $ \norm{\cdot} $, the inner product of $ E $ is $ \inner{\cdot,\cdot} $, and $ \norm{x}_2 $ is $ \sqrt{\inner{x,x}} $. Let $ \Fone $ be the class of $ L- $smooth convex functions, i. e. $ f\in\Fone $ if and only if $ f $ is a differentiable convex function and $ \nabla f $ is $ L- $Lipschitz continuous. Given a convex function $ f $, $ f^*: E^*\to(-\infty,+\infty] $ denotes conjugate function of $ f $, whose function value at $ s\in E^* $ is given by $ f^*(s)=\sup_{x\in E}\{\inner{s,x}-f(x)\}. $ For a convex function $ \varphi $, the Bergman-divergence of $ \varphi $ is \[
	D_\varphi(y,x)=\varphi(y)-\varphi(x)-\inner{\varphi'(x),y-x},
	\]
	above $ \varphi'(x)\in\partial\varphi(x) $. $ \varphi $ is said to be $ \sigma- $strongly convex if $ D_\varphi(y,x)\geqslant\dfrac{\sigma}{2}\norm{y-x}^2. $ The $ x^*\in C $ represents the minimizer of $ f $ over closed convex set $ C $.
	\section{Mirror Descent}
	\label{sec:2}
	The mirror descent with constant step size is given by the recursive rule\[
	x_{k+1}=\argmin_x\left\{s\inner{\nabla f(x_k),x-x_k}+\frac{D_\varphi(x,x_k)}{\sigma}\right\},
	\]
	where $ D_\varphi(y,x) $ is the Bergman divergence respect to $ \sigma- $strongly convex function $ \varphi $ defined on $ C $. To derive ODEs for mirror descent, we need to obtain an explicit iteration rule like gradient descent rather than the recursive rule given by the solution of another optimization problem. One can use the well known theorem that $ \varphi $ is $ \sigma- $strongly convex if and only if $ \varphi^*\in\Fone[\sigma^{-1}] $, and then transform the above iteration rule into the following algorithm from mirror map view \begin{equation}
		\begin{split}
			z_{k+1}&=z_k-\sigma s\nabla f(x_k),\\
			x_{k+1}&=\nabla\varphi^*(z_{k+1}),
		\end{split}
		\label{eq:mirrormapview}
	\end{equation}
	with initial $ \nabla\varphi^*(z_0)=x_0 $. As $ s\to 0 $, we achieve the continuous limit of \eqref{eq:mirrormapview}:\begin{equation}
		\begin{split}
			\Z&=-\sigma\nabla f(X),\\
			X&=\nabla\varphi^*(Z),\\
			Z(0)&=z_0, X(0)=x_0 \text{ with } \nabla\varphi^*(z_0)=x_0.
		\end{split}
		\label{eq:mirrorode}
	\end{equation}
	Next, we construct Lyapunov functions to analyze the dynamics of trajectories of \eqref{eq:mirrorode}. Enlightened by the Lyapunov functions in \cite{shi19} for gradient flow and the Lyapunov function in \cite{krichene2015} for accelerated mirror descent, we propose the following Lyapunov function for \eqref{eq:mirrorode}\begin{equation}
		\E(t)=t\sigma[f(X(t))-f(x^*)]+D_{\varphi^*}(Z(t),z^*),
		\label{eq:lyapunovmirrorc}
	\end{equation}
	where $ \nabla\varphi^*(z^*)=x^* $. The first term can be regraded as the potential term  and the second term is a mix. The significant of mixed term in \eqref{eq:lyapunovmirrorc} is that the derivative of $ Z $ is equal to $ -\nabla f(X) $. It's worthwhile to mention that the mixed term should not be the Bergman-divergence on $ E $. The main reason is that $ \varphi $ isn't differentiable in general, which make $ D_\varphi(\cdot,\cdot) $ ill-defined. Right now, we can investigate the derivative of $ \E(t) $ that play a central role in characterizing properties of \eqref{eq:mirrorode}.
	\begin{theorem}
		\label{thm:1}
		Let $X(t), Z(t)$ be the solution to \eqref{eq:mirrorode}. If $ f\in\Fone $ and $ \varphi $ is twice continuous differentiable, then
		\[
		\EE(t)\leqslant-t\sigma\|\X\|^2.
		\]
	\end{theorem}
	\begin{proof}
		By definition of the Bergman divergence, we have\begin{align*}
			\frac{d}{dt}\mathcal{E}(t)&=\sigma[f(X)-f(x^*)]+t\sigma\inner{\nabla f(X),\X}+\inner{\Z,\nabla\varphi^*(Z)-\nabla\varphi^*(z^*)}\\
			&\leqslant\sigma \inner{\nabla f(X),X-x^*}-t\inner{\Z,\X}-\sigma\inner{\nabla f(X),X-x^*}\\
			&=-t\inner{\nabla^2\varphi(X)\X,\X}\\
			&\leqslant-t\sigma\|\X\|^2.
		\end{align*}
	\end{proof}
	From the above theorem, We can now analysis the convergence rate of trajectories of \eqref{eq:mirrorode}.
	\begin{corollary}
		\label{cor:1}
		Let $X(t), Z(t)$ be the solution to \eqref{eq:mirrorode}. If $ f\in\Fone $ and $ \varphi $ is twice continuous differentiable, then
		\begin{align*}
			f(X(t))-f(x^*)&\leqslant\frac{D_{\varphi^*}(z_0,z^*)}{t\sigma},\\
			\inf_{0\leqslant u\leqslant t}\|\X(u)\|^2&=o\left(\frac{1}{t^2}\right).
		\end{align*}
	\end{corollary}
	\begin{proof}
		By Theorem \ref{thm:1}, we have $ \E $ is a non-increasing function respect to $ t $, which leads to \[
		\sigma [f(X(t))-f(x^*)]\leqslant\E(t)\leqslant\E(0)=D_{\varphi^*}(z_0,z^*).
		\]
		Diving $ \sigma $ on both sides of inequality, we get the convergence rate of function value. 
		
		Inspired by \cite{chen22}, we integrate the inequality in Theorem \ref{thm:1} from $ \frac{t}{2} $ to $ t $, and have\[
		\int_{\frac{t}{2}}^{t}\EE(u)du\leqslant-\sigma\int_{\frac{t}{2}}^{t}u\|\X(u)\|^2du\leqslant-\sigma\inf_{\frac{t}{2}\leqslant u\leqslant t}\|\X(u)\|^2\int_{\frac{t}{2}}^{t}udu=-\frac{3\sigma t^2}{8}\inf_{\frac{t}{2}\leqslant u\leqslant t}\|\X(u)\|^2.
		\]
		Thus\[
		\inf_{0\leqslant u\leqslant t}\|\X(u)\|^2\leqslant\inf_{\frac{t}{2}\leqslant u\leqslant t}\|\X(u)\|^2\leqslant\frac{8}{3\sigma t^2}\left[\E\left(\frac{t}{2}\right)-\E(t)\right].
		\]
		Since the non-negative function $ \E(t) $ is non-increasing on $ (0,+\infty) $, $ \lim\limits_{t\to \infty}\E(t) $ exists. Thus by Cauchy's criteria, we have $ \E\left(\frac{t}{2}\right)-\E(t)=o(1) $. In conclusion, \[
		\inf_{0\leqslant u\leqslant t}\|\X(u)\|^2=o\left(\frac{1}{t^2}\right).
		\]
	\end{proof}
	
	In practice, the conjugate of $ \varphi $ is hard to be calculated, which limits its application to estimate the upper bound of complexity of optimization problem. To overcome such problem, we translate the Bergman-divergence respect to the dual space $ E^* $ into the Bergman-divergence on $ E $ via the following lemma.
	\begin{lemma}
		\label{lem:1}
		If $ \nabla\varphi^*(z_0)=x_0, \nabla\varphi^*(z^*)=x^* $, then \[
		D_{\varphi^*}(z_0,z^*)=D_\varphi(x^*,x_0).
		\]
	\end{lemma}
	\begin{proof}
		Due to Theorem 23.5 in \cite{rockafellar70}, under the assumption of Lemma \ref{lem:1}, we have\begin{align*}
			\varphi(x_0)+\varphi^*(z_0)=\inner{z_0,x_0},\\
			\varphi(x^*)+\varphi^*(z^*)=\inner{z^*,x^*}.
		\end{align*}
		Thus\begin{align*}
			D_{\varphi^*}(z_0,z^*)&=\varphi(x^*)-\varphi(x_0)-\inner{z_0-z^*,x^*}+\inner{z_0,x_0}-\inner{z^*,x^*}\\
			&=\varphi(x^*)-\varphi(x_0)-\inner{z_0,x^*-x_0}\\
			&=D_\varphi(x^*,x_0).
		\end{align*}
	\end{proof}
	With such observation, we have
	\begin{corollary}
		\label{cor:2}
		Let $X(t), Z(t)$ be the solutions to \eqref{eq:mirrorode}. If $ f\in\Fone $ and $ \varphi $ is twice continuous differentiable, then
		\begin{align*}
			f(X(t))-f(x^*)&\leqslant\frac{D_\varphi(x^*,x_0)}{t\sigma}.
		\end{align*}
	\end{corollary}
	
	Now we translate the continuous-time Lyapunov function into discrete-time Lyapunov function to analysis the mirror descent. Motivated by \eqref{eq:lyapunovmirrorc}, we consider the following discrete-time Lyapunov function:\begin{equation}
		\E(k)\overset{\Delta}{=}\E(ks)=ks\sigma[f(x_k)-f(x^*)]+D_{\varphi^*}(z_k,z^*).
		\label{eq:lyapunovmirrord}
	\end{equation}
	Analogous to the continuous-time Lyapunov functions, the difference of discrete-time Lyapunov functions $ \E(k+1)-\E(k) $ is also importance to investigate convergence results of mirror descent. Similar to continuous-time cases, the upper bound of $ \E(k+1)-\E(k) $ is dominated by $ -\norm{x_{k+1}-x_k}^2 $. With such observation, we can conclude the following theorem.
	\begin{theorem}
		\label{thm:2}
		Let $ f\in\mathscr{F}^{1,1}_L, \varphi $ is $ \sigma- $strongly convex. If the step-size $ s\in[0,\dfrac{2}{L}) $, then the sequences $ \{z_k\}, \{x_k\} $ generated by \eqref{eq:mirrormapview} satisfies\begin{align*}
			f(x_k)-f(x^*)\leqslant\frac{D_\varphi(x^*,x_0)}{k\sigma s},\\
			\min_{0\leqslant i\leqslant k}\norm{x_{i+1}-x_i}^2=o\left(\frac{1}{k^2}\right).
		\end{align*}
	\end{theorem}
	The details of proof can be found in \ref{sec:thm2}. As we can see, Theorem \ref{thm:2} is the discrete version of Theorem \ref{thm:1} with time $ t_k=ks $ expect the assumptions they required.
	\section{Accelerated Mirror Descent}
	\label{sec:3}
	\subsection{High-resolution ODE for accelerated mirror descent}
	The accelerated mirror descent proposed in \cite{nesterov05} with fixed step-size $ s $ is given by the recursive rule
	\[
	\begin{split}
		y_{k+1}&=\argmin_y\left\{s\inner{\nabla f(x_k),y-x_k}+\frac{1}{2}\norm{x-x_k}^2\middle|y\in C\right\},\\
		z_{k+1}&=\argmin_z\left\{\frac{1}{\sigma s}\varphi(z)+\sum_{i=0}^{k}\frac{i+1}{2}[f(x_i)+\inner{\nabla f(x_i),z-x_i}]\middle|z\in C\right\},\\
		x_{k+1}&=\frac{2}{k+3}z_{k+1}+\frac{k+1}{k+3}y_{k+1}.
	\end{split}
	\]
	To build up the ODEs for it, we need to derive explicit solutions of optimization algorithms appear above under additional conditions. The solution of gradient map can be solve explicit if $ C=E $ and the norm is Euclidean norm. Under these assumption, we have\begin{align*}
		z_{k+1}&=\argmin_z\left\{\frac{1}{\sigma s}\varphi(z)+\sum_{i=0}^{k}\frac{i+1}{2}[f(x_i)+\inner{\nabla f(x_i),z-x_i}]\right\},\\
		x_{k+1}&=\frac{2}{k+3}z_{k+1}+\frac{k+1}{k+3}[x_k-s\nabla f(x_k)].
	\end{align*}
	Next we use the mirror-map to derive the explicit formula for $ z_{k+1} $. The first-order optimal condition for $ z_{k+1} $ is\[
	-\sigma s\sum_{i=0}^{k}\frac{i+1}{2}\nabla f(x_i)\in\partial\varphi(z_{k+1}).
	\] 
	Let $ \tilde{z}_{k+1}=-\sigma s\sum_{i=0}^k\frac{i+1}{2}\nabla f(x_i) $. The above relation can be rewritten as $ \tilde{z}_{k+1}\in\partial\varphi(z_{k+1}) $, or $ z_{k+1}=\nabla\varphi^*(\tilde{z}_{k+1}) $ owing to that $ \varphi^* $ is differentiable and its derivative is the inverse of $ \partial\varphi $. Noticing that $ \tilde{z}_{k+1}-\tilde{z}_k=-\sigma s\nabla f(x_k) $, we have the following reformed iteration rule.
	\begin{equation}
		\label{eq:unconstrained}
		\begin{split}
			\tilde{z}_{k+1}-\tilde{z}_k&=-\sigma s\frac{k+1}{2}\nabla f(x_k),\\
			(k+3)x_{k+1}&=2\nabla\varphi^*(\tilde{z}_{k+1})+(k+1)[x_k-s\nabla f(x_k)].
		\end{split}
	\end{equation}
	
	Replace $ \tilde{z}_{k+1} $ by $ z_{k+1} $, and slightly modify the iteration rule, we have\begin{equation}
		\begin{split}
			\frac{z_{k+1}-z_k}{\sqrt{s}}&=-\sigma\frac{(k+1)\sqrt{s}}{2}\nabla f(x_k),\\
			(k+1)\sqrt{s}\frac{x_{k+1}-x_k}{\sqrt{s}}&=2\nabla\varphi^*(z_{k+1})-2x_{k+1}-[(k+1)\sqrt{s}]\sqrt{s}\nabla f(x_k).
		\end{split}
		\label{eq:acceleratedmirrorodeform}
	\end{equation}
	\eqref{eq:acceleratedmirrorodeform} can be seen as applying forward-backward Euler scheme to the high-resolution ODE:\begin{equation}
		\begin{split}
			\Z&=-\frac{t\sigma}{2}\nabla f(X),\\
			\X&=\frac{2}{t}\nabla\varphi^*(Z)-\frac{2}{t}X-\underbrace{\sqrt{s}\nabla f(X)}_{\text{gradient correction term}},\\
			X(0)&=x_0, Z(0)=z_0 \text{ with } \nabla\varphi^*(z_0)=x_0, \X(0)=0,
		\end{split}
		\label{eq:acceleratedmirrorode}
	\end{equation}
	with time $ t_k=(k+1)\sqrt{s} $ and forward step to gradient correction term. Before taking a deep look at the behavior of \eqref{eq:acceleratedmirrorode}, we discuss the existence and uniqueness of the solution to \eqref{eq:acceleratedmirrorode}. In the previous work \cite{krichene2015}, they prove the existence and uniqueness of the solution to low-resolution ODEs by considering the equi-Lipschitz-continuous and uniformly bounded family  of solution $ (X_\delta, Z_\delta)|_{[0,t_0]} $ to \[
	\begin{split}
		\Z&=-\frac{t}{r}\nabla f(X),\\
		\X&=\frac{r}{\max\{t,\delta\}}(\nabla\varphi^*(Z)-X),\\
		X(0)&=x_0, Z(0)=z_0 \text{ with } \nabla\varphi^*(z_0)=x_0, \X(0)=0.
	\end{split}
	\]
	Due to Arzel\`{a}-Ascoli theorem, as $ \delta\to 0 $, the solution of low-resolution is obtained. In high-resolution ODEs case, we only introduce the term $ \sqrt{s}\nabla f(X) $, which is Lipschitz continuous respect to $ X $ and ensure the existence of $ (X_\delta, Z_\delta) $. The rest of the proof is analogous to the previous discussion. 
	
	Now we establish the convergence rate of \eqref{eq:acceleratedmirrorode}. First, we consider the Lyapunov function\begin{equation}
		\E(t)=t^2\sigma[f(X)-f(x^*)]+4D_{\varphi^*}(Z,z^*).
		\label{eq:lyapunovmirrorac}
	\end{equation}
	Before estimate the differentiability of $ \E(t) $, we first remark that we prefer the following ODEs form
	\begin{equation}
		\begin{split}
			\Z&=-\frac{t}{2}\nabla f(X),\\
			\nabla\varphi^*(Z)&=\frac{t}{2}\X+\frac{t}{2}\sqrt{s}\nabla f(X)+X,
		\end{split}
		\label{eq:acceleratedmirrorode2}
	\end{equation}
	because the information of $ \nabla\varphi^*(Z) $ plays a central role in estimate the upper bound of $ \EE(t) $. The upper bound of $\E(t)$ is given as followed.
	\begin{theorem}
		\label{thm:3}
		Let $ X(t), Z(t) $ be the solutions to \eqref{eq:acceleratedmirrorode}, then\[
		\EE(t)\leqslant-{t^2\sigma}\norm{\nabla f(X)}^2_2.
		\]
	\end{theorem}
	\begin{proof}
		The upper bound of derivative of \eqref{eq:lyapunovmirrorac} can be estimated by the following inequalities.\begin{align*}
			\EE(t)=&2t\sigma[f(X)-f(x^*)]+t^2\sigma\inner{\nabla f(X),\X}+4\inner{\Z,\nabla\varphi^*(Z)-\nabla\varphi^*(z^*)}\\
			\leqslant&2t\sigma\inner{\nabla f(X),X-x^*}+{t^2\sigma}\inner{\nabla f(X),\X}-\sigma\inner{t\nabla f(X),t\X+2X-2x^*+t\nabla f(X)}\\
			=&-t^2\sigma\norm{\nabla f(X)}^2_2.
		\end{align*}
	\end{proof}
	With Theorem \ref{thm:3}, we can characterize properties of trajectories of \eqref{eq:acceleratedmirrorode2} by using technique similar to the proof of corollary \ref{cor:1}.
	\begin{corollary}
		\label{cor:3}
		Let $ X(t), Z(t) $ be the solutions to \eqref{eq:acceleratedmirrorode}, then\begin{align*}
			f(X(t))-f(x^*)&\leqslant\frac{4D_\varphi(x^*,x_0)}{t^2\sigma},\\
			\inf_{0\leqslant u\leqslant t}\norm{\nabla f(X(u))}_2^2&=o\left(\frac{1}{t^3}\right).
		\end{align*}
	\end{corollary}
	\begin{proof}
		Since $ \E(t) $ is non-increasing, we have\[
		t^2\sigma[f(X(t))-f(x^*)]\leqslant\E(t)\leqslant\E(0)=4D_{\varphi^*}(z_0,z^*)=4D_\varphi(x^*,x_0).
		\]
		The last equality is due to Lemma \ref{lem:1}. Thus we reach the convergence rate of function value. For the convergence rate of gradient norm, we integrate the inequality in Theorem \ref{thm:3} for $ \frac{t}{2} $ to $ t $, and get\[
		\int_{\frac{t}{2}}^{t}\EE(u)du\leqslant-\int_{\frac{t}{2}}^{t}t^2\sigma\norm{\nabla f(X(u))}^2_2du\leqslant-\frac{7t^3\sigma}{24}\inf_{\frac{t}{2}\leqslant u\leqslant t}\norm{\nabla f(X(u))}_2^2.
		\]
		Or\[
		\inf_{\frac{t}{2}\leqslant u\leqslant t}\norm{\nabla f(X(u))}^2_2\leqslant\frac{24}{7t^3\sigma}\left[\E\left(\frac{t}{2}\right)-\E(t)\right]
		\]
		Since the non-negative value function $ \E(t) $ is non-increasing on $ (0,+\infty) $, $ \lim\limits_{t\to\infty}\E(t) $ exists. Thus by Cauchy's criteria, we have $ \E\left(\frac{t}{2}\right)-\E(t)=o(1) $, thus\[
		\inf_{0\leqslant u\leqslant t}\norm{\nabla f(X(u))}^2_2\leqslant\inf_{\frac{t}{2}\leqslant u\leqslant t}\norm{\nabla f(X(u))}^2_2=o\left(\frac{1}{t^3}\right).
		\]
	\end{proof}
	\subsection{The discrete cases with $ C=E $ and Euclidean norm}
	This section is devoted to analyze the convergence rates of \eqref{eq:acceleratedmirror} under the assumption that $ C=E $ equipped with Euclidean norm, i. e. the recursive rule described in \eqref{eq:unconstrained}. It can be seen as the non-Euclidean extension of Nesterov's accelerated gradient method, which inspired us to generalize the Lyapunov function in \cite{shi21} for NAG-C with more refined convergence results. First, we construct the following Lyapunov function\begin{equation}
		\E(k)=(k+1)(k+2)\sigma s[f(x_k)-f(x^*)]+4D_{\varphi^*}(z_{k+1},z^*).
	\end{equation}
	The upper bound of the difference $ \E(k+1)-\E(k) $ is estimated by the following theorem.
	\begin{theorem}
		\label{thm:4}
		Let the sequences $ \{x_k\}, \{z_k\} $ generated by the recursive rule\begin{align*}
			z_{k+1}&=z_k-\frac{(k+1)\sigma s}{2}\nabla f(x_k),\\
			\nabla\varphi^*(z_{k+1})&=\frac{1}{2}\left[(k+3)x_{k+1}-(k+1)x_k+(k+1)s\nabla f(x_k)\right].
		\end{align*}
		If $ f\in\Fone $, $ \varphi $ is $ \sigma- $strongly convex function on $ \mathbb{R}^n $ and $ 0<s\leqslant\frac{1}{L} $, then\[
		\E(k+1)-\E(k)\leqslant-\frac{(k+1)(k+2)\sigma s^2}{2}\norm{\nabla f(x_k)}^2.
		\]
	\end{theorem}
	The details of the proof can be found in \ref{sec:thm4}. With such important estimation, we have the following result.
	\begin{corollary}
		If $ f\in\Fone $, $ \varphi $ is $ \sigma- $strongly convex function on $ \mathbb{R}^n $ and $ 0<s\leqslant\frac{1}{L} $, then the sequences $ \{x_k\}, \{z_k\} $ generated by \eqref{eq:unconstrained} satisfy\begin{align*}
			f(x_k)-f(x^*)&\leqslant\frac{2\sigma s[f(x_0)-f(x^*)]+4D_{\varphi^*}(z_1,z^*)}{(k+1)(k+2)\sigma s},\\
			\min_{0\leqslant i\leqslant k}\norm{\nabla f(x_i)}^2_2&=o\left(\frac{1}{k^3}\right).
		\end{align*}
	\end{corollary}
	\begin{proof}
		Theorem \ref{thm:4} implies that $ \{\E(k)\} $ is non-increasing, which implies\[
		(k+1)(k+2)\sigma s[f(x_k)-f(x^*)]\leqslant\E(k)\leqslant\E(0)=2\sigma s[f(x_k)-f(x^*)]+4D_{\varphi^*}(z_1,z^*).
		\]
		Diving $(k+1)(k+2)\sigma s$ on both sides of the above inequality, we obtain the required inequality. Next, summing up the inequalities in Theorem \ref{thm:4} from $ \left\lfloor\dfrac{ k}{2}\right\rfloor $ to $ k $, we have\[
		-\sum_{i=\left\lfloor\frac{k}{2}\right\rfloor}^{k}\frac{(i+1)(i+2)\sigma s^2}{2}\norm{\nabla f(x_i)}^2\leqslant\sum_{i=\left\lfloor\frac{k}{2}\right\rfloor}^k \E(i+1)-\E(i)=\E(k+1)-\E\left(\left\lfloor\frac{k}{2}\right\rfloor\right).
		\]
		Thus\[
		\frac{\sigma s^2}{6}\min_{\left\lfloor\frac{k}{2}\right\rfloor\leqslant i\leqslant k}\norm{\nabla f(x_i)}^2_2\leqslant\dfrac{\E(k+1)-\E\left(\left\lfloor\frac{k}{2}\right\rfloor\right)}{(k+1)(k+2)(k+3)-\left\lfloor\frac{k}{2}\right\rfloor\left(\left\lfloor\frac{k}{2}\right\rfloor+1\right)\left(\left\lfloor\frac{k}{2}\right\rfloor+2\right)}
		\]
		Since the non-negative sequence $ \{\E(k)\} $ is non-increasing, $ \lim\limits_{k\to\infty}\E(k) $ exists. By Cauchy's criteria, $ \E(k+1)-\E\left(\left\lfloor\frac{k}{2}\right\rfloor\right)=o(1) $. In conclusion, we have\[
		\min_{0\leqslant i\leqslant k}\norm{\nabla f(x_i)}^2_2\leqslant\min_{\frac{\lfloor k\rfloor}{2}\leqslant i\leqslant k}\norm{\nabla f(x_i)}_2^2=o\left(\frac{1}{k^3}\right).
		\]
	\end{proof}
	\subsection{The general discrete cases}
	Now we analysis the original Nesterov's accelerated mirror descent via Lyapunov function framework. Motivated by the convergence result proven in \cite{nesterov05} and the continuous-time Lyapunov function, we build up the following Lyapunov function. 
	\begin{equation}
		\E(k)=k(k+1)\sigma s[f(y_k)-f(x^*)]+4D_{\varphi^*}(z_k,z^*).
	\end{equation}
	The upper bound of the difference $ \E(k+1)-\E(k) $ is given by the below theorem.
	
	\begin{theorem}
		\label{thm:5}
		Let $ f\in\Fone, \varphi $ be $ \sigma- $stongly convex function defined on $ C $. Let $\{x_k\}, \{y_k\}, \{z_k\}$ be the sequences generated by
		\begin{align*}
			z_{k+1}&=z_k-\frac{(k+1)\sigma s}{2}\nabla f(x_k),\\
			y_{k+1}&=\mathop{\arg\min}_{y}\left\{s\inner{\nabla f(x_k),y}+\frac{1}{2}\norm{y-x_k}^2\middle|y\in C\right\},\\
			x_{k+1}&=\frac{k+1}{k+3}y_{k+1}+\frac{2}{k+3}\nabla\varphi^*(z_{k+1}),
		\end{align*}
		then \[
		\E(k+1)-\E(k)\leqslant-\frac{(k+1)(k+2)(1-Ls)}{2}\norm{y_{k+1}-x_k}^2.
		\]
	\end{theorem}
	The details of the proof can be found in \ref{sec:thm5}. Now we can characterize the convergence results of \eqref{eq:acceleratedmirror}.
	\begin{corollary}
		\label{cor:5}
		If $ f\in\Fone, \varphi $ is $ \sigma- $stongly convex function defined on $ C $, let $\{x_k\}, \{y_k\}, \{z_k\}$ be the sequences generated by \eqref{eq:acceleratedmirror}. If $ 0<s\leqslant\dfrac{1}{L} $, then $ \{x_k\}, \{y_k\} $ satisfy\begin{align*}
			f(y_k)-f(x^*)\leqslant\frac{4D_\varphi(x^*,x_0)}{k(k+1)\sigma s}.
		\end{align*}
		In addition, $ 0<s<\dfrac{1}{L} $, $ \{x_k\}, \{y_k\} $ satisfy\[
		\inf_{0\leqslant u\leqslant k}\norm{y_{k+1}-x_k}^2=o\left(\frac{1}{k^3}\right).
		\]
	\end{corollary}
	\begin{proof}
		From Theorem \ref{thm:5}, we know that $ 0<s\leqslant\dfrac{1}{L} $ is a sufficient condition for $ \E(k+1)-\E(k)\leqslant 0 $. Thus\[
		k(k+1)\sigma s[f(y_k)-f(x^*)]\leqslant\E(k)\leqslant\E(0)=4D_\varphi(x^*,x_0).
		\]
		By dividing $k(k+1)\sigma s$ on both sides, we achieve the desired inequality. In addition, if we sum up the inequalities in Theorem \ref{thm:5} from $ \left\lfloor\frac{ k}{2}\right\rfloor $ to $ k $, then we have\[
		\sum_{i=\left\lfloor\frac{ k}{2}\right\rfloor}^k\frac{(i+1)(i+2)(1-Ls)}{2}\norm{y_{i+1}-x_i}^2\leqslant\E(k+1)-\E\left(\left\lfloor\frac{ k}{2}\right\rfloor\right).
		\]
		Thus\[
		\min_{\left\lfloor\frac{ k}{2}\right\rfloor\leqslant i\leqslant k}\norm{y_{i+1}-x_i}^2_2\leqslant\dfrac{6\E(k+1)-6\E\left(\frac{\lfloor k\rfloor}{2}\right)}{\left[(k+1)(k+2)(k+3)-\left\lfloor\frac{ k}{2}\right\rfloor\left(\left\lfloor\frac{k}{2}\right\rfloor+1\right)\left(\left\lfloor\frac{ k}{2}\right\rfloor+2\right)\right](1-Ls)}.
		\]
		Since the non-negative sequence $ \{\E(k)\} $ is non-increasing, $ \lim\limits_{k\to\infty}\E(k) $ exists. By Cauchy's criteria, $ \E(k+1)-\E\left(\left\lfloor\frac{ k}{2}\right\rfloor\right)=o(1) $. In conclusion\[
		\min_{0\leqslant i\leqslant k}\norm{y_{i+1}-x_i}^2_2\leqslant\min_{\left\lfloor\frac{ k}{2}\right\rfloor\leqslant i\leqslant k}\norm{y_{i+1}-x_i}_2^2=o\left(\frac{1}{k^3}\right).
		\]
	\end{proof}
	\subsection{Higher-order mirror descent}
	In \cite{wibisono16}, Wibisono et al. proposed a higher-order mirror descent, which is a generalization of Nesterov's accelerated cubic Newton method. The higher-order mirror descent is described as followed:
	\begin{equation}
		\begin{split}
			x_{k+1}&=\frac{p}{k+p}z_k+\frac{k}{k+p}y_k,\\
			z_{k+1}&=\argmin_z\left\{Cps(k+1)^{(p-1)}\inner{\nabla f(y_{k+1}),z}+D_\varphi(z,z_k)\right\}.
		\end{split}
		\label{eq:wibisono}
	\end{equation}
	Where $ p\geqslant 2 $ is an integer, $ k^{(p)}=k(k+1)\cdots(k+p-1) $ and the sequence $ \{y_k\} $ satisfies\begin{equation}\label{eq:acceleratekey}
		\inner{\nabla f(y_k),x_k-y_k}\geqslant Ms^{\frac{1}{p-1}}\norm{\nabla f(y_k)}_*^{\frac{p}{p-1}},
	\end{equation}
	for some constant $ M>0 $. Also the Bergman divergence $ D_\varphi $ is required to be lower bounded by $ p- $th power of of the norm\begin{equation}\label{eq:pthstrong}
		D_\varphi(x,y)\geqslant\frac{\sigma}{p}\norm{x-y}^p,\quad\forall x, y\in C.
	\end{equation}
	Or we say $ \varphi $ is $ \sigma- $strongly convex of order $ p $. \cite{wibisono16} proves that if $ f $ is $ p- $times continuous differentiable and $ \nabla^{(p-1)}f $ is $ L-$Lipschitz continuous, then the $ y_k $ generated by the $ (p-1)- $th order gradient map\[
	y_k=\argmin_y\left\{\sum_{i=0}^{p-1}\frac{1}{i!}\nabla^i f(x)(y-x)^i+\frac{N}{sp}\norm{y-x}^p\middle|y\in C\right\}
	\]
	satisfies \eqref{eq:acceleratekey} for some chosen $ N $. They used estimate-sequence technique to show that if $0<C\leqslant\dfrac{M^{p-1}}{p^p}$, $f(x_k)-f(x^*)\leqslant O\left(\dfrac{1}{sk^p}\right)$. In this paper, we exploit Lyapunov function framework to demonstrate the same convergence rate for function values of higher-order mirror descent with a wider range of parameters. Additionally, we show that the convergence rate of gradient norm minimization is $o\left(\dfrac{1}{k^{p+1}}\right)$. To make use of Lyapunov function framework, we need to translate the recursive rule into the following mirror-map form.
	\begin{equation}
		\begin{split}
			x_{k+1}&=\frac{p}{k+p}z_k+\frac{k}{k+p}y_k,\\
			z_{k+1}&=z_k-Csp(k+1)^{(p-1)}\nabla f(y_{k+1}).
		\end{split}
		\label{eq:wibisonomirrormap}
	\end{equation}
	Before proving the convergence result of higher-order mirror descent, the smoothness of $ \varphi^* $ is required.
	\begin{lemma}
		\label{lem:2}
		If $ \varphi $ is $ \sigma- $strongly convex of order $ p $, then $\varphi^*$ is differentiable. In addition, we have
		\[
		\varphi^*(z')\leqslant\varphi^*(z)+\inner{\nabla\varphi^*(z),z'-z}+\frac{p-1}{p}\left(\frac{1}{\sigma}\right)^{\frac{1}{p-1}}\norm{z'-z}^{\frac{p}{p-1}},\quad\forall z, z'\in\mathrm{dom}(\varphi). 
		\] 
	\end{lemma}
	The proof of Lemma \ref{lem:2} is referred to the proof of Lemma 4.2.2 in \cite{nesterov18}. Finally, we can describe and prove the convergence rate of \eqref{eq:wibisonomirrormap}.
	\begin{theorem}
		\label{thm:7}
		Let $ f $ be a continuous differentiable convex function, $ \varphi $ is $ \sigma- $strongly convex of order $ p $ and $ \{x_k\}, \{y_k\}, \{z_k\} $ be the sequences generated by \eqref{eq:wibisonomirrormap} and \eqref{eq:acceleratekey}. If $ 0<C\leqslant\dfrac{\sigma M^{p-1}}{(p-1)^{p-1}p} $, then\[
		f(y_k)-f(x^*)\leqslant\frac{D_\varphi(x,x^*)}{Csk^{(p)}}.
		\]
		In addition, if $ 0<C<\dfrac{\sigma M^{p-1}}{(p-1)^{p-1}p} $, then\[
		\min_{0\leqslant i\leqslant k}\norm{\nabla f(y_k)}^{\frac{p}{p-1}}_*=o\left(\frac{1}{k^{p+1}}\right).
		\]
	\end{theorem}
	This theorem tells us that Lyapunov function framework is superior to estimate-sequence technique. The convergence result in \cite{wibisono16} is hold when $ 0<C\leqslant\dfrac{\sigma M^{p-1}}{p^p} $, by contrast the assumption in theorem \ref{thm:7} is $ 0<C\leqslant\dfrac{\sigma M^{p-1}}{(p-1)^{p-1}p} $, which is wider than the previous one. Also estimate-sequence technique cannot tells us any information about the convergence rate of gradient norm, compared to $ \min_{0\leqslant i\leqslant k}\norm{\nabla f(y_k)}^{\frac{p}{p-1}}_*=o\left(\frac{1}{k^{p+1}}\right) $ in theorem \ref{thm:7}.
	
	\section{Discussion}
	In this paper, we generalized the second-order high-resolution ODEs for Nesterov’s accelerated gradient method	to first-order high-resolution ODEs for Nesterov’s accelerated mirror descent. Also, we presented Lyapunov function framework for characterizing the convergence rates of accelerated schemes in both continuous-time and discrete-time.	Finally, we generalized the Lyapunov function framework to higher-order mirror descent.
	
	There are a good many open problems of high-resolution ODEs. First of all, there remains a significant number of algorithms that have not yet been transformed into discretizations of ordinary differential equations. The excessive gap technique proposed by Nesterov for solving primal-dual problems is one such example.	In recent times, researchers have increasingly focused on optimization algorithms that involve solving optimization	subproblems, such as alternative direction method of multipliers and augmented Lagrangian methods. It would be	worthwhile to analyze these algorithms using ODEs and attempt to derive their accelerated	variants.
	
	\appendix
	\section{The Proof in Section \ref{sec:2}}
	\subsection{The Proof of Theorem \ref{thm:2}}
	\label{sec:thm2}
	Recall the Lyapunov function\[
	\E(k)=k\sigma s[f(x_k)-f(x^*)]+D_{\varphi^*}(z_k,z^*).
	\]
	
	Step 1: Estimate the upper bound of difference $ \E(k+1)-\E(k) $. Diving the difference into three parts.
	\begin{align*}
		\E(k+1)-\E(k)=&\underbrace{(k+1)\sigma s(f(x_{k+1})-f(x_k))}_{\text{I}}+\underbrace{\sigma s(f(x_k)-f(x^*))}_{\text{II}}\\
		&+\underbrace{\varphi^*(z_{k+1})-\varphi^*(z_k)+\inner{z_{k+1}-z_k,\nabla\varphi^*(z^*)}}_{\text{III}}.
	\end{align*}
	Since $ f\in\Fone, $ we have the following inequalities.\begin{align*}
		f(x_{k+1})-f(x_k)&\leqslant\inner{\nabla f(x_k),x_{k+1}-x_k}+\frac{L}{2}\norm{x_{k+1}-x_k}^2.\\
		f(x_k)-f(x^*)&\leqslant\inner{\nabla f(x_k),x_k-x^*}.
	\end{align*}
	Thus\begin{align*}
		\text{I}&\leqslant \underbrace{(k+1)\sigma s\inner{\nabla f(x_k),x_{k+1}-x_k}}_{\text{I}_1}+\frac{(k+1) L\sigma s}{2}\norm{x_{k+1}-x_k}^2\\
		\text{II}&\leqslant \underbrace{\sigma s\inner{\nabla f(x_k),x_k-x^*}}_{\text{II}_1}.
	\end{align*}
	By $ \varphi^*\in\Fone[\sigma^{-1}] $, we have the following inequality\[
	\varphi^*(z_{k+1})-\varphi^*(z_k)\leqslant\inner{z_{k+1}-z_k,\nabla\varphi^*(z_{k+1})}-\frac{\sigma}{2}\norm{\nabla\varphi^*(z_{k+1})-\nabla\varphi^*(z_k)}^2.
	\]
	Thus\begin{align*}
		\text{III}\leqslant-\underbrace{\sigma s\inner{\nabla f(x_k),x_{k+1}-x^*}}_{\text{III}_1}-\frac{\sigma}{2}\norm{x_{k+1}-x_k}^2.
	\end{align*}
	By summing up previous estimation, we have\begin{align*}
		\E(k+1)-\E(k)\leqslant&\underbrace{(k+1)\sigma s\inner{\nabla f(x_k),x_{k+1}-x_k}}_{\text{I}_1}+\frac{(k+1) L\sigma s}{2}\norm{x_{k+1}-x_k}^2+\underbrace{\sigma s\inner{\nabla f(x_k),x_k-x^*}}_{\text{II}_1}\\
		&-\underbrace{\sigma s\inner{\nabla f(x_k),x_{k+1}-x^*}}_{\text{III}_1}-\frac{\sigma}{2}\norm{x_{k+1}-x_k}^2.
	\end{align*}
	First, we have\[
	\text{II}_1-\text{III}_1=-\sigma s\inner{\nabla f(x_k),x_{k+1}-x_k}.
	\]
	Thus
	\begin{align*}
		\text{I}_1+\text{II}_1-\text{III}_1&=k\sigma s\inner{\nabla f(x_k),x_{k+1}-x_k}\\
		&=-k\inner{z_{k+1}-z_k,\nabla\varphi^*(z_{k+1})-\nabla\varphi^*(z_k)}\\
		&\leqslant-k\sigma\norm{\nabla\varphi^*(z_{k+1})-\nabla\varphi^*(z_k)}^2\\
		&=-k\sigma\norm{x_{k+1}-x_k}^2.
	\end{align*}
	The inequalities is because $ \varphi^*\in\Fone[\sigma^{-1}]. $ In conclusion, we have\begin{align*}
		\E(k+1)-\E(k)\leqslant&\left(-k-\frac{1}{2}+\frac{(k+1)Ls}{2}\right)\sigma\norm{x_{k+1}-x_k}^2.
	\end{align*}
	Step 2: From Lyapunov function to convergence rate. As long as $ s\in[0,\frac{1}{L}] $, $ \E(k+1)-\E(k)\leqslant -\frac{k\sigma}{2}\norm{x_{k+1}-x_k}^2 $. Then we have\[
	k\sigma s[f(x_k)-f(x^*)]\leqslant\E(k)\leqslant\E(0)=D_{\varphi^*}(z_0,z^*)=D_\varphi(x^*,x_0).
	\]
	The last equality is due to Lemma \ref{lem:1}. Diving $k\sigma s$ on both sides, we attain the desired inequality. 
	
	If $ 0<s\leqslant\frac{1}{L} $, then we sum up $ \E(i+1)-\E(i) $ from $ \left\lfloor\frac{k}{2}\right\rfloor $ to $ k $, and have\[
	\E(k+1)-\E\left(\left\lfloor\frac{k}{2}\right\rfloor\right)=\sum_{i=\left\lfloor\frac{k}{2}\right\rfloor}^{k}\E(i+1)-\E(i)\leqslant\sum_{i=\left\lfloor\frac{k}{2}\right\rfloor}^k\left(\frac{Ls}{2}-1\right)i\sigma\norm{x_{i+1}-x_i}^2.
	\]
	Thus\[
	\min_{0\leqslant i\leqslant k}\norm{x_{i+1}-x_i}^2\leqslant\min_{\left\lfloor\frac{k}{2}\right\rfloor\leqslant i\leqslant k}\norm{x_{i+1}-x_i}^2\leqslant\frac{4\E(k+1)-4\E(\left\lfloor\frac{k}{2}\right\rfloor)}{\left[k(k+1)-\left(\left\lfloor\frac{k}{2}\right\rfloor-1\right)\left\lfloor\frac{k}{2}\right\rfloor\right](2-Ls)}.
	\]
	Since the non-negative sequence $ \{\E(k)\} $ is non-increasing, $ \lim\limits_{k\to\infty}\E(k) $ exists. By Cauchy's criteria, $ \E(k+1)-\E\left(\left\lfloor\frac{k}{2}\right\rfloor\right)=o(1) $. In conclusion\[
	\min_{0\leqslant i\leqslant k}\norm{x_{i+1}-x_i}^2=o\left(\frac{1}{k^2}\right).
	\]
	\section{The Proof in Section \ref{sec:3}}
	\label{sec:b}
	\subsection{The Proof of Theorem \ref{thm:4}}
	\label{sec:thm4}
	Recall the following Lyapunov functions\[
	\E(k)=(k+1)(k+2)\sigma s[f(x_k)-f(x^*)]+4D_{\varphi^*}(z_{k+1},z^*).
	\]
	Step 1: Dividing the difference $ \E(k+1)-\E(k) $ into three parts. By showing the difference $ \E(k+1)-\E(k) $, we have\begin{align*}
		\E(k+1)-\E(k)=&\underbrace{(k+1)(k+2)\sigma s[f(x_{k+1})-f(x_k)]}_{\text{I}}+\underbrace{(2k+4)\sigma s[f(x_{k+1})-f(x^*)]}_{\text{II}}\\
		&+\underbrace{4(\varphi^*(z_{k+2})-\varphi^*(z_{k+1})+\inner{\nabla\varphi^*(z^*),z_{k+2}-z_{k+1}})}_{\text{III}}.
	\end{align*}
	Step 2: Estimate the upper bound of three parts. Since $ f\in\Fone $, we have\[
	f(x_{k+1})-f(x_k)\leqslant\inner{\nabla f(x_{k+1}),x_{k+1}-x_k}-\frac{1}{2L}\norm{\nabla f(x_{k+1})-\nabla f(x_k)}^2_2.
	\]
	Thus the upper bound of I can be estimated by\[
	\text{I}\leqslant \underbrace{(k+1)(k+2)\sigma s\inner{\nabla f(x_{k+1}),x_{k+1}-x_k}}_{\text{I}_1}-\underbrace{\frac{(k+1)(k+2)\sigma s}{2L}\norm{\nabla f(x_{k+1})-\nabla f(x_k)}^2_2}_{\text{I}_2}.
	\]
	Also due to $ f\in\Fone $, we have\begin{align*}		
		f(x_{k+1})-f(x^*)&\leqslant\inner{\nabla f(x_{k+1}),x_{k+1}-x^*}-\frac{1}{2L}\norm{\nabla f(x_{k+1})}^2_2\\
		&\leqslant\inner{\nabla f(x_{k+1}),x_{k+1}-x^*}-\frac{s}{2}\norm{\nabla f(x_{k+1})}^2_2.
	\end{align*}
	Thus 
	\begin{align*}
		\text{II}\leqslant\underbrace{(2k+4)\sigma s\inner{\nabla f(x_{k+1}),x_{k+1}-x^*}}_{\text{II}_1}-\underbrace{(k+2)\sigma s^2\norm{\nabla f(x_{k+1})}^2_2}_{\text{II}_2}.
	\end{align*}
	Since $ \varphi^*\in\Fone[\sigma^{-1}] $ and $ z_{k+2}-z_{k+1}=-\frac{(k+2)\sigma s}{2}\nabla f(x_{k+1}) $, the upper bound of III is
	\begin{align*}
		\text{III}&\leqslant 4\inner{z_{k+2}-z_{k+1},\nabla\varphi^*(z_{k+1})-\nabla\varphi^*(z^*)}+\frac{2}{\sigma}\norm{z_{k+2}-z_{k+1}}^2_2\\
		&=\underbrace{4\inner{z_{k+2}-z_{k+1},\nabla\varphi^*(z_{k+1})-\nabla\varphi^*(z^*)}}_{\text{III}_1}+\underbrace{\frac{(k+2)^2\sigma s^2}{2}\norm{\nabla f(x_{k+1})}^2_2}_{\text{III}_2}.
	\end{align*}
	
	Step 3: Estimate the upper bound of $ \E(k+1)-\E(k) $. Combining $ \text{I}, \text{II}_1 $ and $ \text{III}_1 $, we have\begin{align*}
		\E(k+1)-\E(k)\leqslant&\underbrace{(k+1)(k+2)\sigma s\inner{\nabla f(x_{k+1}),x_{k+1}-x_k}}_{\text{I}_1}-\underbrace{\frac{(k+1)(k+2)\sigma s}{2L}\norm{\nabla f(x_{k+1})-\nabla f(x_k)}^2_2}_{\text{I}_2}\\
		&+\underbrace{(2k+4)\sigma s\inner{\nabla f(x_{k+1}),x_{k+1}-x^*}}_{\text{II}_1}-\underbrace{(k+2)\sigma s^2\norm{\nabla f(x_{k+1})}^2_2}_{\text{II}_2}\\
		&+\underbrace{4\inner{z_{k+2}-z_{k+1},\nabla\varphi^*(z_{k+1})-\nabla\varphi^*(z^*)}}_{\text{III}_1}+\underbrace{\frac{(k+2)^2\sigma s^2}{2}\norm{\nabla f(x_{k+1})}^2_2}_{\text{III}_2}.
	\end{align*}
	
	First, consider $ \text{I}_1+\text{II}_1 $.\begin{align*}
		\text{I}_1+\text{II}_2=&\sigma s\inner{(k+2)\nabla f(x_{k+1}),(k+1)(x_{k+1}-x_k)}+\sigma s\inner{(k+2)\nabla f(x_{k+1}),2(x_{k+1}-x^*)}\\
		=&\sigma s\inner{(k+2)s\nabla f(x_{k+1}),2\nabla\varphi^*(z_{k+1})-(k+1)s\nabla f(x_k)-2x_{k+1}}+\sigma s\inner{(k+2)\nabla f(x_{k+1}),2(x_{k+1}-x^*)}\\
		=&\sigma s\inner{(k+2)s\nabla f(x_{k+1}),2\nabla\varphi^*(z_{k+1})-2\nabla\varphi^*(z^*)}-(k+1)(k+2)\sigma s^2\inner{\nabla f(x_{k+1}),\nabla f(x_k)}.
	\end{align*}
	Since $ z_{k+2}-z_{k+1}=-\dfrac{(k+2)\sigma s}{2}\nabla f(x_{k+1}), $ we have\[
	\text{III}_1=-\sigma s\inner{(k+2)\nabla f(x_{k+1}),2\nabla\varphi^*(z_{k+1})-2\nabla\varphi^*(z^*)}
	\]
	Thus\[
	\text{I}_1+\text{II}_2+\text{III}_1=-(k+1)(k+2)\sigma s^2\inner{\nabla f(x_{k+1}),\nabla f(x_k)}.
	\]
	Due to $ 0<s\leqslant\frac{1}{L} $, we have
	\begin{align*}
		-\text{I}_2\leqslant&-\frac{(k+1)(k+2)\sigma s^2}{2}\norm{\nabla f(x_{k+1})-\nabla f(x_k)}^2_2\\
		=&-\frac{(k+1)(k+2)\sigma s^2}{2}\left[\norm{\nabla f(x_{k+1})}^2_2+\norm{\nabla f(x_k)}^2\right]+(k+1)(k+2)\sigma s^2\inner{\nabla f(x_{k+1}),\nabla f(x_k)}.
	\end{align*}
	Now adding $ \text{I}_1+\text{II}_2+\text{III}_3 $ on both sides of the above inequality, we have\[
	\text{I}_1+\text{II}_2+\text{III}_3-\text{I}_2=-\frac{(k+1)(k+2)\sigma s^2}{2}\left[\norm{\nabla f(x_{k+1})}^2_2+\norm{\nabla f(x_k)}_2^2\right].
	\]
	The upper bound of $ \E(k+1)-\E(k) $ is given as followed.\begin{align*}
		\E(k+1)-\E(k)\leqslant&\left(-\frac{(k+1)(k+2)}{2}-(k+2)+\frac{(k+2)^2}{2}\right)\sigma s^2\norm{\nabla f(x_{k+1})}^2\\
		&-\frac{(k+1)(k+2)\sigma s^2}{2}\norm{\nabla f(x_k)}^2\\
		\leqslant&-\frac{(k+1)(k+2)\sigma s^2}{2}\norm{\nabla f(x_k)}^2.
	\end{align*}
	\subsection{The Proof of Theorem \ref{thm:5}}
	\label{sec:thm5}
	Recall the Lyapunov function:
	\[
	\E(k)=k(k+1)\sigma s[f(y_k)-f(x^*)]+4D_{\varphi^*}(z_k,z^*).
	\]
	Step 1: Diving the difference $ \E(k+1)-\E(k) $ into four terms. 
	\begin{align*}
		\E(k+1)-\E(k)=&\underbrace{(k+1)(k+2)\sigma s[f(y_{k+1})-f(x_k)]}_{\text{I}}+\underbrace{k(k+1)\sigma s[f(x_k)-f(y_k)]}_{\text{II}}\\
		&+\underbrace{(2k+2)\sigma s[f(x_k)-f(x^*)]}_{\text{III}}+\underbrace{4[\varphi^*(z_{k+1})-\varphi^*(z_k)+\inner{\nabla\varphi^*(z^*),z_{k+1}-z_k}]}_{\text{IV}}.
	\end{align*}
	Step 2: Estimate the upper bound of all four terms.
	$ f\in\mathscr{F}^{1,1}_L $,\begin{align*}
		s[f(y_{k+1})-f(x_k)]\leqslant&s\inner{\nabla f(x_k),y_{k+1}-x_k}+\frac{Ls}{2}\norm{y_{k+1}-x_k}^2\\
		=&s\inner{\nabla f(x_k),y_{k+1}-x_k}+\frac{1}{2}\norm{y_{k+1}-x_k}^2-\frac{1-Ls}{2}\norm{y_{k+1}-x_k}^2\\
		=&\min_y\{s\inner{\nabla f(x_k),y-x_k}+\frac{1}{2}\norm{y-x_k}|y\in C\}-\frac{1-Ls}{2}\norm{y_{k+1}-x_k}^2\\
		\leqslant&s\inner{\nabla f(x_k),\frac{k}{k+2}y_k+\frac{2}{k+2}\nabla\varphi^*(z_{k+1})-x_k}\\
		&+\frac{1}{2}\norm{\frac{k}{k+2}y_k-x_k+\frac{2}{k+2}\nabla\varphi^*(z_{k+1})}^2-\frac{1-Ls}{2}\norm{y_{k+1}-x_k}^2\\
		=&s\inner{\nabla f(x_k),\frac{k}{k+2}y_k+\frac{2}{k+2}\nabla\varphi^*(z_{k+1})-x_k}\\&
		+\frac{2}{(k+2)^2}\norm{\nabla\varphi^*(z_{k+1})-\nabla\varphi^*(z_k)}^2-\frac{1-Ls}{2}\norm{y_{k+1}-x_k}^2.
	\end{align*}
	By the definition of $ y_{k+1} $, we have\[
	s\inner{\nabla f(x_k),y_{k+1}-x_k}+\frac{1}{2}\norm{y_{k+1}-x_k}^2=\min_{y\in C}\{s\inner{\nabla f(x_k),y-x_k}+\frac{1}{2}\norm{y-x_k}^2\}.
	\]
	Since $ y_k\in C, \nabla\varphi^*(z_{k+1})\in C $, the convex combination of then belongs to $ C $, thus
	\begin{align*}
		&s\inner{\nabla f(x_k),y_{k+1}-x_k}+\frac{1}{2}\norm{y_{k+1}-x_k}^2\\
		\leqslant& s\inner{\nabla f(x_k),\frac{k}{k+2}y_k+\frac{2}{k+2}\nabla\varphi^*(z_{k+1})-x_k}+\frac{1}{2}\norm{\frac{k}{k+2}y_k-x_k+\frac{2}{k+2}\nabla\varphi^*(z_{k+1})}^2.
	\end{align*}
	Since $ x_k=\dfrac{k}{k+2}y_k+\dfrac{2}{k+2}\nabla\varphi^*(z_k) $, we have $ \frac{k}{k+2}y_k-x_k=-\frac{2}{k+2}\nabla\varphi^*(z_k) $, thus
	\begin{align*}
		s[f(y_{k+1})-f(x_k)]\leqslant& s\inner{\nabla f(x_k),\frac{k}{k+2}y_k+\frac{2}{k+2}\nabla\varphi^*(z_{k+1})-x_k}\\
		&\frac{2}{(k+2)^2}\norm{\nabla\varphi^*(z_{k+1})-\nabla\varphi^*(z_k)}^2-\frac{1-Ls}{2}\norm{y_{k+1}-x_k}^2.
	\end{align*}
	Thus the upper bound of I is given by\begin{align}
		\text{I}\leqslant& \underbrace{(k+1)\sigma s\inner{\nabla f(x_k),ky_k+2\nabla\varphi^*(z_{k+1})-(k+2)x_k}}_{\text{I}_1}\\
		&+\underbrace{2\sigma\norm{\nabla\varphi^*(z_{k+1})-\nabla\varphi^*(z_k)}^2}_{\text{I}_2}-\frac{(k+1)(k+2)(1-Ls)\sigma}{2}\norm{y_{k+1}-x_k}^2.
	\end{align}
	The above inequality uses a simple estimation $ \frac{k+1}{k+2}<1 $, which simplifies the successive calculation.
	
	Also due to the convexity of $ f $, the upper bound of II, III can be estimated by\begin{align*}
		\text{II}&\leqslant \underbrace{(k+1)\sigma s\inner{\nabla f(x_k),kx_k-ky_k}}_{\text{II}_1},\\
		\text{III}&\leqslant \underbrace{(k+1)\sigma s\inner{\nabla f(x_k),2x_k-2x^*}}_{\text{III}_1}.
	\end{align*}
	
	Since $ \varphi^*\in\Fone[\sigma^{-1}] $, we have\[
	\text{IV}\leqslant \underbrace{4\inner{z_{k+1}-z_k,\nabla\varphi^*(z_{k+1})-\nabla\varphi^*(z^*)}-2\sigma\norm{\nabla\varphi^*(z_{k+1})-\nabla\varphi^*(z_k)}}_{\text{IV}_1}
	\]
	Step 3: Deduce the upper bound of $ \E(k+1)-\E(k) $. Summing up all the upper bound of four terms, we have
	\begin{align*}
		\E(k+1)-\E(k)\leqslant&\underbrace{(k+1)\sigma s\inner{\nabla f(x_k),ky_k+2\nabla\varphi^*(z_{k+1})-(k+2)x_k}}_{\text{I}_1}+\underbrace{2\sigma\norm{\nabla\varphi^*(z_{k+1})-\nabla\varphi^*(z_k)}^2}_{\text{I}_2}\\
		&+\underbrace{(k+1)\sigma s\inner{\nabla f(x_k),kx_k-ky_k}}_{\text{II}_1}+\underbrace{(k+1)\sigma s\inner{\nabla f(x_k),2x_k-2x^*}}_{\text{III}_1}\\
		&+\underbrace{4\inner{z_{k+1}-z_k,\nabla\varphi^*(z_{k+1})-\nabla\varphi^*(z^*)}-2\sigma\norm{\nabla\varphi^*(z_{k+1})-\nabla\varphi^*(z_k)}}_{\text{IV}_1}\\
		&-\frac{(k+1)(k+2)(1-Ls)\sigma}{2}\norm{y_{k+1}-x_k}^2.
	\end{align*}
	First, consider $ \text{I}_1+\text{II}_1+\text{III}_1 $. The sum of these three terms is 
	\begin{align*}
		\text{I}_1+\text{II}_1+\text{III}_1&=(k+1)\sigma s\inner{\nabla f(x_k),2\nabla\varphi^*(z_{k+1})-2x^*}.
	\end{align*}
	Since $ (k+1)\sigma s\nabla f(x_k)=-2(z_{k+1}-z_k), \nabla\varphi^*(z^*)=x^* $, we have\[
	\text{I}_1+\text{I}_2+\text{II}_1+\text{III}_1+\text{IV}_1=0.
	\]
	In conclusion,\[
	\E(k+1)-\E(k)\leqslant-\frac{(k+1)(k+2)(1-Ls)\sigma}{2}\norm{y_{k+1}-x_k}^2.
	\]
	
	\subsection{The Proof of Theorem \ref{thm:7}}
	\label{sec:thm7}
	Step 1: Consider the Lyapunov function\[
	\E(k)=Csk^{(p)}[f(y_k)-f(x^*)]+D_{\varphi^*}(z_k,z^*).
	\]
	Step 2: Dividing the difference $ \E(k+1)-\E(k) $ into 3 pieces.\begin{align*}
		\E(k+1)-\E(k)=&\underbrace{Csk^{(p)}[f(y_{k+1})-f(y_k)]}_{\text{I}}+\underbrace{Csp(k+1)^{(p-1)}[f(y_{k+1})-f(x^*)]}_{\text{II}}\\
		&+\underbrace{\varphi^*(z_{k+1})-\varphi^*(z_k)+\inner{\nabla\varphi^*(z^*),z_{k+1}-z_k}}_{\text{IV}}.
	\end{align*}
	Step 3: Estimate the upper bound of 3 pieces. Because of the convexity of $ f $, we have\begin{align*}
		\text{I}&\leqslant \underbrace{Csk^{(p)}\inner{\nabla f(y_{k+1}),y_{k+1}-y_k}}_{\text{I}_1},\\
		\text{II}&\leqslant\underbrace{Csp(k+1)^{(p-1)}\inner{\nabla f(y_{k+1}),y_{k+1}-x^*}}_{\text{II}_1}.
	\end{align*}
	Utilize lemma \ref{lem:2}, we have\begin{align*}
		\text{III}\leqslant&\inner{z_{k+1}-z_k,\nabla\varphi^*(z_k)-\nabla\varphi^*(z^*)}+\frac{p-1}{p}\left(\frac{1}{\sigma}\right)^{\frac{1}{p-1}}\norm{z_{k+1}-z_k}^{\frac{p}{p-1}}_*\\
		=&-\underbrace{Csp(k+1)^{(p-1)}\inner{\nabla f(y_{k+1}),\nabla\varphi^*(z_k)-\nabla\varphi^*(z^*)}}_{\text{III}_1}+\underbrace{(p-1)\left(\frac{p}{\sigma}\right)^{\frac{1}{p-1}}[Cs(k+1)^{(p-1)}]^\frac{p}{p-1}\norm{\nabla f(y_{k+1})}^{\frac{p}{p-1}}_*}_{\text{III}_2}.
	\end{align*}
	Step 4: Calculate the upper bound of $ \E(k+1)-\E(k) $. By summing up the previous result, we have\begin{align*}
		\E(k+1)-\E(k)\leqslant&\underbrace{Csk^{(p)}\inner{\nabla f(y_{k+1}),y_{k+1}-y_k}}_{\text{I}_1}+\underbrace{Csp(k+1)^{(p-1)}\inner{\nabla f(y_{k+1}),y_{k+1}-x^*}}_{\text{II}_1}\\
		&-\underbrace{Csp(k+1)^{(p-1)}\inner{\nabla f(y_{k+1}),\nabla\varphi^*(z_k)-\nabla\varphi^*(z^*)}}_{\text{III}_1}\\&+\underbrace{(p-1)\left(\frac{p}{\sigma}\right)^{\frac{1}{p-1}}[Cs(k+1)^{(p-1)}]^\frac{p}{p-1}\norm{\nabla f(y_{k+1})}^{\frac{p}{p-1}}_*}_{\text{III}_2}.
	\end{align*}
	First, we consider the sum of first three terms.\begin{align*}
		\text{I}_1+\text{II}_1=&Cs(k+1)^{(p-1)}\inner{\nabla f(y_{k+1}),(k+p)y_{k+1}-ky_k-px^*}\\
		=&Cs(k+1)^{(p)}\inner{\nabla f(y_{k+1}),y_{k+1}-x_{k+1}}+Cs(k+1)^{(p-1)}\inner{\nabla f(y_{k+1}),(k+p)x_{k+1}-ky_k-px^*}.
	\end{align*}
	By the basic condition \eqref{eq:acceleratekey} and the recursive rule \eqref{eq:wibisonomirrormap}, we have\begin{align*}
		\text{I}_1+\text{II}_1\leqslant&-CMs^{\frac{p}{p-1}}(k+1)^{(p)}\norm{\nabla f(y_{k+1})}^{\frac{p}{p-1}}_*+Csp(k+1)^{(p-1)}\inner{\nabla f(y_{k+1}),p\nabla\varphi^*(z_k)-p\nabla\varphi^*(z^*)}.
	\end{align*}
	Thus\begin{align*}
		\text{I}_1+\text{II}_1-\text{III}_1\leqslant-KMS^{\frac{p}{p-1}}(k+1)^{(p)}\norm{\nabla f(y_{k+1})}^{\frac{p}{p-1}}_*.
	\end{align*}
	In conclusion, \[
	\E(k+1)-\E(k)\leqslant\left\{-CM(k+1)^{(p)}+(p-1)\left(\frac{p}{\sigma}\right)^{\frac{1}{p-1}}[C(k+1)^{(p-1)}]^{\frac{p}{p-1}}\right\}s^{\frac{p}{p-1}}\norm{\nabla f(y_{k+1})}_*^{\frac{p}{p-1}}.
	\]
	Step 5: Convergence rates of \eqref{eq:wibisonomirrormap}. Since $[(k+1)^{(p-1)}]^{\frac{p}{p-1}}\leqslant (k+1)^{(p)}$, we have\[
	\E(k+1)-\E(k)\leqslant\left[-M+(p-1)\left(\frac{p}{\sigma}\right)^{\frac{1}{p-1}}C^{\frac{1}{p-1}}\right]C(k+1)^{(p)}s^{\frac{p}{p-1}}\norm{\nabla f(y_{k+1})}^{\frac{p}{p-1}}_*.
	\]
	If $0<C\leqslant\dfrac{\sigma M^{p-1}}{(p-1)^{p-1}p}$, then $\E(k+1)-\E(k)\leqslant 0$. Thus\[
	Csk^{(p)}[f(y_k)-f(x^*)]\leqslant\E(k)\leqslant\E(0)=D_\varphi(x^*,x_0).
	\]
	Let $D=Cs^{\frac{p}{p-1}}[M-(p-1)\left(\dfrac{p}{\sigma}\right)^{\frac{1}{p-1}}C^{\frac{1}{p-1}}]$. If $0<C<\dfrac{\sigma M^{p-1}}{(p-1)^{p-1}p}$, then $D>0$. Summing $\E(i+1)-\E(i)$ from $\left\lfloor\dfrac{k}{2}\right\rfloor$ to $k$, we have\begin{align*}
		\E(k+1)-\E\left(\left\lfloor\dfrac{k}{2}\right\rfloor\right) &=\sum_{i=\left\lfloor\frac{k}{2}\right\rfloor}^{k}\E(i+1)-\E(i) \\
		&\leqslant-D\sum_{i=\left\lfloor\frac{k}{2}\right\rfloor}^{k}(i+1)^{(p)}\norm{\nabla f(y_{i+1})}^{\frac{p}{p-1}}_*\\
		&\leqslant-D\min_{0\leqslant i\leqslant k}\norm{\nabla f(y_{i+1})}^{\frac{p}{p-1}}_*\sum_{i=\left\lfloor\frac{k}{2}\right\rfloor}^{k}(i+1)^{(p)}\\
		&=-\frac{D}{p+1}\left[(k+1)^{(p+1)}-\left(\left\lfloor\dfrac{k}{2}\right\rfloor\right)^{(p+1)}\right]\min_{0\leqslant i\leqslant k}\norm{\nabla f(y_{i+1})}^{\frac{p}{p-1}}_*.
	\end{align*}
	Since $\E(k+1)-\E\left(\left\lfloor\dfrac{k}{2}\right\rfloor\right)\to 0$, we have\[
	\min_{0\leqslant i\leqslant k}\norm{\nabla f(y_{i+1})}^{\frac{p}{p-1}}_*=o\left(\frac{1}{k^{p+1}}\right).
	\]
	

\end{document}